\documentclass[reqno,11pt,oneside]{amsart}
\usepackage[whole]{bxcjkjatype} 
\usepackage[unicode]{hyperref} 
\DeclareFontEncoding{OT2}{}{}
\DeclareFontSubstitution{OT2}{cmr}{m}{n}
\DeclareFontFamily{OT2}{cmr}{\hyphenchar\font45 }
\DeclareFontShape{OT2}{cmr}{m}{n}{
<5><6><7><8><9>gen*wncyr
<10><10.95><12><14.4><17.28><20.74><24.88>wncyr10}{}
\DeclareFontShape{OT2}{cmr}{b}{n}{
<5><6><7><8><9>gen*wncyb
<10><10.95><12><14.4><17.28><20.74><24.88>wncyb10}{}
\DeclareMathAlphabet{\mathcyr}{OT2}{cmr}{m}{n}
\DeclareMathAlphabet{\mathcyb}{OT2}{cmr}{b}{n}
\SetMathAlphabet{\mathcyr}{bold}{OT2}{cmr}{b}{n}
\usepackage{fullpage}
\usepackage{fancybox,ascmac,comment,xcolor}
\usepackage{amssymb,mathrsfs}
\usepackage{amsfonts}
\usepackage{amsmath}
\usepackage{amsthm}
\usepackage{mathtools}
\usepackage{enumerate}

\usepackage{mleftright}
\mleftright
\usepackage[all]{xy}
\usepackage{tikz}
\usepackage{tikz-cd}
\mathtoolsset{showonlyrefs=true}
\numberwithin{equation}{section}
\makeatletter
\newcommand{\shortmathcal}[1]{\@tfor\ch:=#1\do{
\expandafter\edef\csname c\ch\endcsname{\noexpand\mathcal{\ch}}
}}
\newcommand{\shortmathbb}[1]{\@tfor\ch:=#1\do{
\expandafter\edef\csname bb\ch\endcsname{\noexpand\mathbb{\ch}}
}}
\newcommand{\shortmathbf}[1]{\@tfor\ch:=#1\do{
\expandafter\edef\csname b\ch\endcsname{\noexpand\mathbf{\ch}}
}}
\newcommand{\shortboldsymbol}[1]{\@tfor\ch:=#1\do{
\expandafter\edef\csname bs\ch\endcsname{\noexpand\boldsymbol{\ch}}
}}
\newcommand{\shortmathfrak}[1]{\@tfor\ch:=#1\do{
\expandafter\edef\csname f\ch\endcsname{\noexpand\mathfrak{\ch}}}}
\newcommand{\shortmathscr}[1]{\@tfor\ch:=#1\do{
\expandafter\edef\csname s\ch\endcsname{\noexpand\mathscr{\ch}}}}
\newcommand{\shortmathrm}[1]{\@tfor\ch:=#1\do{
\expandafter\edef\csname r\ch\endcsname{\noexpand\mathrm{\ch}}
}}
\let\@@span\span
\def\sp@n{\@@span\omit\advance\@multicnt\m@ne}
\makeatother
\shortmathcal{ABCDEFGHIJKLMNOPQRSTUVWXYZ}
\shortmathbb{ABCDEFGHIJKLMNOPQRSTUVWXYZ}
\shortmathbf{ABCDEFGHIJKLMNOPQRSTUVWXYZabcdefghijklmnopqrstuvwxyz}
\shortboldsymbol{ABCDEFGHIJKLMNOPQRSTUVWXYZabcdefghijklmnopqrstuvwxyz}
\shortmathfrak{ABCDEFGHIJKLMNOPQRSTUVWXYZabcdefghjklmnopqrstuvwxyz}%小文字iは絶対に入れない!
\shortmathscr{ABCDEFGHIJKLMNOPQRSTUVWXYZ}
\shortmathrm{ABCDEFGHIJKLMNOPQRSTUVWXYZabcdefghjklmnopqrstuvwxyz} %iとりました
\DeclareFontEncoding{OT2}{}{}
\DeclareFontSubstitution{OT2}{cmr}{m}{n}
\DeclareFontFamily{OT2}{cmr}{\hyphenchar\font45 }
\DeclareFontShape{OT2}{cmr}{m}{n}{
<5><6><7><8><9>gen*wncyr
<10><10.95><12><14.4><17.28><20.74><24.88>wncyr10}{}
\DeclareFontShape{OT2}{cmr}{b}{n}{
<5><6><7><8><9>gen*wncyb
<10><10.95><12><14.4><17.28><20.74><24.88>wncyb10}{}
\DeclareMathAlphabet{\mathcyr}{OT2}{cmr}{m}{n}
\DeclareMathAlphabet{\mathcyb}{OT2}{cmr}{b}{n}
\SetMathAlphabet{\mathcyr}{bold}{OT2}{cmr}{b}{n}

\newcommand{\id}{\mathrm{id}}

\newcommand{\ri}{\mathsf{ri}}

\newcommand{\emp}{\varnothing}
\newcommand{\ep}{\varepsilon}

\renewcommand{\mod}{\mathrm{mod}}

\newcommand{\sh}{\mathbin{\mathcyr{sh}}}

\newcommand{\BIMU}{\mathsf{BIMU}}
\newcommand{\MU}{\mathsf{MU}}
\newcommand{\LU}{\mathsf{LU}}
\newcommand{\mmu}{\mathsf{mu}}

\newcommand{\ganit}{\mathsf{ganit}}

\newcommand{\axit}{\mathsf{axit}}
\newcommand{\amit}{\mathsf{amit}}
\newcommand{\anit}{\mathsf{anit}}
\newcommand{\arit}{\mathsf{arit}}

\newcommand{\ilat}{\mathsf{ilat}}

\newcommand{\gari}{\mathsf{gari}}

\newcommand{\gira}{\mathsf{gira}}

\newcommand{\preari}{\mathsf{preari}}

\newcommand{\preila}{\mathsf{preila}}

\newcommand{\ari}{\mathsf{ari}}

\newcommand{\ila}{\mathsf{ila}}

\newcommand{\GARI}{\mathsf{GARI}}

\newcommand{\ARI}{\mathsf{ARI}}

\newcommand{\adari}{\mathsf{adari}}

\newcommand{\expari}{\mathsf{expari}}

\newcommand{\invmu}{\mathsf{invmu}}

\newcommand{\invgari}{\mathsf{invgari}}

\newcommand{\Pa}{\mathsf{Pa}}
\renewcommand{\Pi}{\mathsf{Pi}}

\newcommand{\foz}{\mathfrak{oz}}

\newcommand{\pic}{\mathsf{pic}}

\newcommand{\fess}{\mathfrak{ess}}

\newcommand{\pir}{\mathsf{pir}}

\newcommand{\as}{\mathsf{as}}
\newcommand{\al}{\mathsf{al}}

\newcommand{\il}{\mathsf{il}}

\newcommand{\di}{\mathsf{di}}

\newcommand{\fre}{\mathfrak{re}}

\newcommand{\dro}{\mathfrak{r}\ddot{\fo}}
\newcommand{\fSe}{\mathfrak{Se}}

\newcommand{\fTe}{\mathfrak{Te}}

\newcommand{\ful}[1]{{}_{#1}\lceil}
\newcommand{\fll}[1]{{}_{#1}\lfloor}
\newcommand{\fur}[1]{\rceil_{#1}}
\newcommand{\flr}[1]{\rfloor_{#1}}

\renewcommand{\neg}{\mathsf{neg}}
\newcommand{\swap}{\mathsf{swap}}
\newcommand{\anti}{\mathsf{anti}}
\newcommand{\pari}{\mathsf{pari}}
\newcommand{\push}{\mathsf{push}}

\newcommand{\mantar}{\mathsf{mantar}}

\newcommand{\swamu}{\mathsf{swamu}}

\newcommand{\leng}{\mathsf{leng}}
\newcommand{\der}{\mathsf{der}}

\newcommand{\ucQ}{\overline{\cQ}}
\newcommand{\lcQ}{\underline{\cQ}}
\newcommand{\ucG}{\overline{\cG}}
\newcommand{\lcG}{\underline{\cG}}
\newcommand{\ucL}{\overline{\cL}}
\newcommand{\lcL}{\underline{\cL}}
\newcommand{\lpreiha}{\mathbin{\underline{\circ}}}
\newcommand{\upreiha}{\mathbin{\underline{\circ}_{\mathsf{ari}}}}
\newcommand{\ls}{\mathfrak{ls}}
\newcommand{\dmr}{\mathfrak{dmr}}

\newcommand{\baral}{\underline{\mathsf{al}}}
\newcommand{\baril}{\underline{\mathsf{il}}}
\newcommand{\mpar}{\mathsf{par}}

\newcommand{\dara}{\mathsf{dara}}
\newcommand{\diri}{\mathsf{diri}}
\newcommand{\bari}{\underline{\ari}}
\newcommand{\DMR}{\mathsf{DMR}}
\newcommand{\ucst}{\bu\text{-}\mathsf{cst}}
\newcommand{\vcst}{\bv\text{-}\mathsf{cst}}
\newcommand{\ds}{\mathfrak{ds}}

\theoremstyle{definition}
\newtheorem{theorem}{Theorem}[section]

\newtheorem{lemma}[theorem]{Lemma}
\newtheorem{corollary}[theorem]{Corollary}

\newtheorem{definition}[theorem]{Definition}
\theoremstyle{remark}
\newtheorem{remark}[theorem]{Remark}

\allowdisplaybreaks
\title{Ecalle's dimorphic transportation and Brown's lifting}
\author{Hanamichi Kawamura}
\address[Hanamichi Kawamura]{Department of Mathematics, Graduate School of Science, Tokyo University of Science, 1-3 Kagurazaka, Shinjuku-ku, Tokyo, 162-8601, Japan}
\email{1125512@ed.tus.ac.jp}
\allowdisplaybreaks
\begin{document}
\begin{abstract}
    Brown's lifting procedure shows that one can solve the double shuffle equations modulo products from solutions of linearized ones. In this paper, we reveal that it is described in the framework of Ecalle's dimorphic transportation.
\end{abstract}
\maketitle
\section{Introduction}
We have been interested in the double shuffle structure in the framework $\DMR_{0}$ as a group and $\dmr_{0}$ as a Lie algebra.
They are originally formulated in a non-commutative power series (or polynomial) ring, but there is another setting which is constructed in the so-called \emph{flexion theory} established by Ecalle.
In particular, it is important to investigate a concrete solution of the double shuffle equation modulo products ($=$ an element of $\dmr_{0}$).
Brown \cite{brown17} got some systems of such solutions, and Matthes--Tasaka \cite{mt19} started to compare Brown's and Ecalle's solutions.
Furthermore, recently Furusho--Hirose--Komiyama \cite{fhk26} revealed that one of Brown's system $\{\psi_{2n+1}\}_{n}$ can be exactly described as Ecalle's \emph{polar singulates}.
On the other hand, Brown \cite{brown17} and Ecalle \cite{ecalle11} also independently obtained a method to get solutions $\chi_{B}(f),\chi_{E}(f)\in\ds_{\cQ}$ of the double shuffle equations admitting some poles by lifting each solution $f\in\ls_{\cQ}$ of the ``linearlized'' double shuffle equations.
In this note, we reveal that Brown's lifting $\ls_{\cQ}\ni f\mapsto\chi_{B}(f)\in\ds_{\cQ}$ essentially coincides with the operator $\adari(\mpar)$, which is an example of Ecalle's lifting.

\begin{theorem}\label{thm:main}
    We have
    \[\swap\circ\anti\circ\chi_{B}=\adari(\mpar)\circ\swap\circ\anti.\]
\end{theorem}

While the original preprint \cite{brown17} does not include the proof that $\chi_{B}(f)$ gives a solution in $\ds_{\cQ}$ for each $f\in\ls_{\cQ}$, Kimura--Tasaka's report \cite{tasaka} present that assertion as a open conjecture. We can give an affirmative answer to it by Theorem \ref{thm:main} and the machinery of Ecalle's dimorphic transportation.

\begin{corollary}\label{cor:main}
    For any $f\in\ls_{\cQ}$, we have $\chi_{B}(f)\in\ds_{\cQ}$.
\end{corollary}

\section*{Acknowledgements}
The author would like to thank Ayu Kawamura for her helpful suggestions on improving the terminology and exposition of this paper.

\section{Matthes--Tasaka's setting and flexions}
\subsection{Some sets of functions}
We basically adopt conventions of Matthes--Tasaka \cite{mt19} and the author's series \cite{kawamura25}, \cite{kawamura252}. The deck $\cF$ should be chosen to include the mutually-conjugate pair of flexion units
\[\Pa\binom{u_{1}}{v_{1}}\coloneqq\frac{1}{u_{1}},\quad\text{and}\quad\Pi\binom{u_{1}}{v_{1}}\coloneqq\frac{1}{v_{1}}.\]
The ground field $\bbK$ is taken as $\bbQ$ for comparing with Brown's work, but we note here that all flexion-theoretic arguments done in this note work in any ground field $\bbK$ with $\mathrm{ch}(\bbK)=0$. As we did in \cite{kawamura252}, the information about $\cF$ and $R$ is often omitted from several notations unless it matters.
We denote by $\BIMU^{\vcst}$ (resp.~$\BIMU^{\ucst}$) the subset of $\BIMU(R)$ consisting of all bimoulds which are indendent of variables in the lower layer (resp.~the upper layer).

\begin{definition}
    \begin{enumerate}
        \item Let $\cQ$ be the direct sum $\bigoplus_{k\in\bbZ}\cQ_{k}$ of the space
    \[\cQ_{k}\coloneqq\left\{\left.f=(f^{(r)})_{r}\in\prod_{r\ge 0}\bbQ(x_{1},\ldots,x_{r})~\right|~\deg(f^{(r)})=k-r~(r\ge 0)\right\}.\]
    Note that $\bbQ(x_{1},\ldots,x_{0})\coloneqq\bbQ$ and the element $0$ has any degree.
    \item The subset $\cG$ (resp.~$\cL$) of $\cQ$ is defined to be the set of all elements having $1$ (resp.~$0$) as its $r=0$ component.
    \item We prepare a copy $\ucQ$ (resp.~$\lcQ$) of $\cQ$ and regard it as a subset of $\BIMU^{\vcst}(\bbQ)$ (resp.~$\BIMU^{\ucst}(\bbQ)$) by the embedding determined by $x_{i}\mapsto u_{i}$ (resp.~$x_{i}\mapsto v_{i}$).
    \item Similarly we denote by $\ucG$, $\lcG$, $\ucL$ and $\lcL$ suitable copies of $\cG$ and $\cL$ being compatible with the definition of $\ucQ$ and $\ucL$.
    \end{enumerate}
\end{definition}

By the identification above, we say that an element $f=(f^{(r)})_{r}\in\ucQ$ (or $\lcQ$) has a component $f^{(r)}$ \emph{of length $r$}, which can have terms of several degrees. Note that $\ucQ$ and $\lcQ$ are not completed with respect to this length filtration, but it does not matter since we always discuss after regarding them as subsets of $\BIMU$. By definition, we have
\begin{gather}
    \ucG=\MU\cap\ucQ,\quad \ucL=\LU\cap\ucQ,\\
    \lcG=\MU\cap\lcQ,\quad \lcL=\LU\cap\lcQ.
\end{gather}

The two types of musical accidentals
    \begin{align}
        f^{\sharp}(x_{1},\ldots,x_{r})&\coloneqq f(x_{1},x_{1}+x_{2},\ldots,x_{1}+\cdots+x_{r}),\\
        f^{\flat}(x_{1},\ldots,x_{r})&\coloneqq f(x_{1},x_{2}-x_{1},\ldots,x_{r}-x_{r-1})    
    \end{align}
    in \cite[\S 2.4]{mt19}, are understood as maps $\sharp\colon\lcQ\to\ucQ$ and $\flat\colon\ucQ\to\lcQ$, respectively.

\begin{definition}[Ihara bracket; {\cite[Definition 4.1]{mt19}}]
    Let $f$ and $g$ be elements of $\lcL$.
    Then we define the Ihara bracket as the commutator
    \[\{f,g\}\coloneqq f\lpreiha g-g\lpreiha f\]
    of the Ihara action $\lpreiha$ given by the bilinearity and
    \begin{multline}
        (f^{(r)}\lpreiha g^{(s)})(x_{1},\ldots,x_{r+s})\\
        \coloneqq\sum_{i=0}^{s}f^{(r)}(x_{i+1}-x_{i},\ldots,x_{i+r}-x_{i})g^{(s)}(x_{1},\ldots,x_{i},x_{i+r+1},\ldots,x_{r+s})\\
+(-1)^{r}\sum_{i=1}^{s}f^{(r)}(x_{i+r}-x_{i+r-1},\ldots,x_{i+r}-x_{i})g^{(s)}(x_{1},\ldots,x_{i-1},x_{i+r},\ldots,x_{r+s}).
    \end{multline}
\end{definition}

The following assertion is obvious by definition.
\begin{lemma}\label{lem:lpreiha_ilat}
    For any $g\in\lcL$, its Ihara action is given by
    \[f\mapsto f\lpreiha g=\axit((\neg\circ\pari\circ\anti)(f),f)(g)+\mmu(f,g).\]
\end{lemma}

Here we also introduce the framework of double shuffle equations used in \cite{brown17} and \cite{mt19}. The symbol $\sh$ denotes the usual shuffle product for variable sequences. We usually use the convention that an element of $\cQ$ behaves linearly for a linear combination of variable sequences. For example, we write
\[f((x_{1},x_{2})\sh (x_{3}))=f(x_{1},x_{2},x_{3})+f(x_{1},x_{3},x_{2})+f(x_{3},x_{1},x_{2}).\]
The harmonic product $\ast$ is defined to be a similar $\bbQ(x_{1},x_{2},\ldots)$-bilinear product which determined by the rules $\emp\sh\bx=\bx\sh\emp=\bx$ and
    \[((x_{i})\uplus\bx)\ast ((x_{j})\uplus\bx')=(x_{i})\uplus(\bx\ast ((x_{j})\uplus\bx'))+(x_{j})\uplus(((x_{i})\uplus\bx)\ast\bx')+\frac{1}{x_{i}-x_{j}}(\bx\ast\bx'),\]
where $\uplus$ denotes the concatenation of sequences.
\begin{definition}
    We define the subspace $\ls_{\cQ}\subseteq\lcL$ to be the set of every function $f\in\lcL$ satisfying\footnote{As pointed out in \cite{tasaka}, in the definition of the space $\mathfrak{pdmr}$ in \cite{brown17}, which should coincide with our $\ds_{\cQ}$, the parity condition (1) is missing.}
    \begin{enumerate}
        \item $f^{(1)}(x_{1})=f^{(1)}(-x_{1})$,
        \item $f^{(r+s)}((x_{1},\ldots,x_{r})\sh(x_{r+1},\ldots,x_{r+s}))=0$ for $r,s\ge 1$,
        \item $(f^{(r+s)})^{\sharp}((x_{1},\ldots,x_{r})\sh(x_{r+1},\ldots,x_{r+s}))=0$ for $r,s\ge 1$.
    \end{enumerate}
    Similarly, we denote by $\ds_{\cQ}\subseteq\lcL$ the subspace consisting of all function $f$ satisfying (1), (3) and
    \begin{equation}\label{eq:il}
        f^{(r+s)}((x_{1},\ldots,x_{r})\ast(x_{r+1},\ldots,x_{r+s}))=0
    \end{equation}
    for each $r,s\ge 1$.
\end{definition}

\begin{remark}\label{rem:alil}
By definition, $\ls_{\cQ}$ is equal to the intersection of $\lcL$ with $\ARI_{\baral/\baral}$. Moreover, the space $\ds_{\cQ}$ is $\lcL\cap\swap(\ARI_{\baral/\baril})$, where $\il$ denotes the $\Pi$-alternality ($=$ alternility). The latter coincidence is checked from Komiyama's result \cite{komiyama21} that the condition \eqref{eq:il} is essentially equivalent to the definition of alternility using $\ganit(\pic)$.
\end{remark}

From the above remark, we can show Corollary \ref{cor:main} assuming Theorem \ref{thm:main} as follows. Since the operator $\anti$ preserves alternality and alternility, it is sufficient to show that $(\swap\circ\anti\circ\chi_{B})(f)\in\ARI_{\baral/\baril}$ for any $f\in\ls_{\cQ}$. Since $f$ is an element of $\ARI_{\baral/\baral}$, so is $(\swap\circ\anti)(f)$. Thus, from Theorem \ref{thm:main}, we obtain
\[(\swap\circ\anti\circ\chi_{B})(f)\in\adari(\mpar)(\ARI_{\baral/\baral}).\]
The right-hand side is equal to $\ARI_{\baral/\baril}$, as shown in \cite[Theorem 3.5]{kawamura252}.

\subsection{Flexions}
By definition and the identification we gave in the previous subsection, we see that $\sharp=(\swap\circ\anti)|_{\lcQ}$ and $\flat=(\anti\circ\swap)|_{\ucQ}$ hold. Moreover, we prepare the definition of the $\ila$-bracket
\begin{align}
    \preila(A,B)&\coloneqq\ilat(B)(A)+\mmu(A,B)\coloneqq\axit(B,(\neg\circ\pari\circ\anti)(B))(A),\\
    \ila(A,B)&\coloneqq\preila(A,B)-\preila(B,A),
\end{align}
for $A,B\in\LU$.

\begin{lemma}\label{lem:ihara_ila}
    Let $f,g\in\lcL$. Then we have
    \[\{f,g\}=\anti(\ila(\anti(g),\anti(f))).\]
\end{lemma}
\begin{proof}
    Since $\anti$ is an anti-homomorphism with respect to the $\mmu$-product \cite[Lemma 2.8]{kawamura25}, we have
    \begin{align}
        &\anti(\ila(\anti(g),\anti(f)))\\
        &=\begin{multlined}[t](\anti\circ\ilat(\anti(f))\circ\anti)(g)-(\anti\circ\ilat(\anti(g))\circ\anti)(f)\\+\anti(\mmu(\anti(g),\anti(f)))-\anti(\mmu(\anti(f),\anti(g)))\end{multlined}\\
        &=(\anti\circ\ilat(\anti(f))\circ\anti)(g)-(\anti\circ\ilat(\anti(g))\circ\anti)(f)+\mmu(f,g)-\mmu(g,f),
    \end{align}
    while the Ihara bracket is computed by Lemma \ref{lem:lpreiha_ilat} as
    \begin{align}
        \{f,g\}
        &=f\lpreiha g-g\lpreiha f\\
        &=\axit((\neg\circ\pari\circ\anti)(f),f)(g)-\axit((\neg\circ\pari\circ\anti)(g),g)(f)+\mmu(f,g)-\mmu(g,f).
    \end{align}
    Hence it is sufficient to show the equality
    \begin{equation}\label{eq:ihara_ila_proof}
        (\anti\circ\ilat(\anti(f))\circ\anti)(g)
        =\axit((\neg\circ\pari\circ\anti)(f),f)(g)
    \end{equation}
    because this identity itself and the one after interchanging $f\leftrightarrow g$ yield the desired equality.
    Since $\anti\circ\amit(\anti(A))\circ\anti=\anit(A)$ holds for any $A\in\LU$, we have
    \begin{align}
        &(\anti\circ\ilat(\anti(f))\circ\anti)(g)\\
        &=(\anti\circ\amit(\anti(f))\circ\anti)(g)+(\anti\circ\anit((\neg\circ\pari)(f))\circ\anti)(g)\\
        &=\anit(f)(g)+\amit((\neg\circ\pari\circ\anti)(f))(g)\\
        &=\axit((\neg\circ\pari\circ\anti)(f),f)(g),
    \end{align}
    and it completes the proof.
\end{proof}

Let us here review the auxilirary operation $\upreiha$ and the bracket $\{-,-\}_{\ari}$ treated in \cite{mt19}, by describing how they related to the original flexion structure instead of directly introducing their definition. As refered in \cite[\S 3.2]{mt19}, they coincide the usual $\preari$ and $\ari$-bracket restricted to $\BIMU^{\vcst}$ and $\LU^{\vcst}$ immediately shown as
\[f\upreiha g=\preari(g,f),\quad \{f,g\}_{\ari}=\ari(g,f)\]
for any $f,g\in\ucL$.
We also note that, the space $\cV$ introduced in \cite[\S 4.2]{mt19} coincides with $\ARI_{\mantar}^{\ucst}\coloneqq\ARI_{\mantar}\cap\LU^{\ucst}$, since $\cV$ is defined to be the space of all $f\in\cL$ satisfying in any length
\[0=f(x_{1},\ldots,x_{r})+(-1)^{r}f(x_{1},\ldots,x_{r})=(f-\mantar(f))(x_{1},\ldots,x_{r}).\]
Hence the following corollary gives another proof of \cite[Proposition 4.5]{mt19}.

\begin{corollary}[{\cite[Proposition 4.5]{mt19}}]\label{cor:mt45}
    We have $\{f,g\}_{\ari}^{\flat}=\{f^{\flat},g^{\flat}\}$ for all $f,g\in\cV$. 
\end{corollary}
\begin{proof}
    By Lemma \ref{lem:ihara_ila}, the assertion follows by showing that 
    \begin{equation}\label{eq:mt_45} (\anti\circ\swap)(\ari(G,F))=\anti(\ila(\swap(G),\swap(F)))
\end{equation}
    for all $F,G\in\ARI_{\mantar}$.
    Using Schneps' identities \cite[(2.4.7), (2.4.8)]{schneps15}, namely
    \begin{equation}\label{eq:amit_swap}
        (\swap\circ\amit(\swap(A))\circ\swap)(B)=\amit(A)(B)+\mmu(B,A)-\swap(\mmu(\swap(B),\swap(A)))
    \end{equation}
    and
    \begin{equation}\label{eq:anit_swap}
        (\swap\circ\anit(\swap(A))\circ\swap)(B)=\anit(\push(A))(B),
    \end{equation}
    that are valid for any $A\in\LU$ and $B\in\BIMU$, we see that
    \begin{align}
        &\ila(\swap(G),\swap(F))\\
        &=\ilat(\swap(F))(\swap(G))-\ilat(\swap(G))(\swap(F))+\swap(\swamu(G,F))-\swap(\swamu(F,G))\\
        &=\begin{multlined}[t]\amit(\swap(F))(\swap(G))-\amit(\swap(G))(\swap(F))+\anit((\neg\circ\pari\circ\anti\circ\swap)(F))(\swap(G))\\-\anit((\neg\circ\pari\circ\anti\circ\swap)(G))(\swap(F))+\swap(\swamu(G,F))-\swap(\swamu(F,G))\end{multlined}\\
        &=\begin{multlined}[t]\swap(\amit(F)(G))+\swap(\mmu(G,F))-\swap(\amit(G)(F))-\swap(\mmu(F,G))\\+\swap(\anit((\push\circ\swap\circ\neg\circ\pari\circ\anti\circ\swap)(F))(G))\\-\swap(\anit((\push\circ\swap\circ\neg\circ\pari\circ\anti\circ\swap)(G))(F))\end{multlined}\\
        &=\begin{multlined}\swap(\amit(F)(G))+\swap(\mmu(G,F))-\swap(\amit(G)(F))-\swap(\mmu(F,G))\\-\swap(\anit((\push\circ\swap\circ\neg\circ\mantar\circ\swap)(F))(G))\\+\swap(\anit((\push\circ\swap\circ\neg\circ\mantar\circ\swap)(G))(F))\end{multlined}\\
        &=\begin{multlined}[t]\swap(\amit(F)(G))+\swap(\mmu(G,F))-\swap(\amit(G)(F))-\swap(\mmu(F,G))\\-\swap(\anit(\mantar(F))(G))+\swap(\anit(\mantar(G))(F))\end{multlined}
    \end{align}
    where in the last equality we used the identity
    \[\push=\neg\circ\mantar\circ\swap\circ\mantar\circ\swap.\]
    Since $F$ and $G$ are assumed to be $\mantar$-invariant, the above computation finally veryfies
    \begin{align}
        &\ila(\swap(G),\swap(F))\\
        &=\begin{multlined}[t]\swap(\amit(F)(G))+\swap(\mmu(G,F))-\swap(\amit(G)(F))\\-\swap(\mmu(F,G))-\swap(\anit(F)(G))+\swap(\anit(G)(F))\end{multlined}\\
        &=\swap(\preari(G,F)-\preari(F,G))\\
        &=\swap(\ari(G,F)),
    \end{align}
    which shows \eqref{eq:mt_45}.
\end{proof}

\section{Dilators}
We recall the definition of $\gari$-dilators from \cite[(28)]{ecalle15}.

\begin{definition}[{\cite[(17)]{ecalle15}}]
    For $S\in\MU$, we define the $\gari$-dilator $\di S\in\LU$ of $S$ by the equality
    \[\der(S)=\preari(S,\di S).\]
\end{definition}

Since the linearization of the $\gari$-product gives the $\preari$ operation
\[\gari(A,1+\ep B)=A+\ep\preari(A,B)\quad (\mod~\ep^{2})\]
as shown in \cite[(2.8.4)]{schneps15}, we can say that the dilator is also defined as
\begin{equation}\label{eq:gari_dilator}
1+\ep\di S=\gari(\invgari(S),\exp(\ep\der)(S)).
\end{equation}

This section is devoted to give a proof for the following theorem, which is the most essential part for the proof of the main theorem \ref{thm:main}. For $A\in\LU$, we denote by $\bari(A)$ the adjoint action on the Lie algebra $\ARI$:
\[\bari(A)(B)\coloneqq\ari(A,B)\quad (A,B\in\LU).\]

\begin{theorem}[{\cite[(42)]{ecalle15}}]\label{thm:main2}
    Let $S$ be an element of $\MU$. Then we have the following identities of linear maps $\LU\to\LU$:
    \begin{align}\label{eq:main_theorem}
        \begin{split}
        \adari(S)&=\sum_{s=0}^{\infty}\sum_{r_{1},\ldots,r_{s}\ge 1}\frac{1}{r_{1}(r_{1}+r_{2})\cdots (r_{1}+\cdots+r_{s})}\bari(\leng_{r_{1}}(\di S))\circ\cdots\circ\bari(\leng_{r_{s}}(\di S)),\\
        \adari(S)&=\sum_{s=0}^{\infty}\sum_{r_{1},\ldots,r_{s}\ge 1}\frac{1}{r_{1}(r_{1}+r_{2})\cdots (r_{1}+\cdots+r_{s})}\bari(\leng_{r_{s}}(\di \ri S))\circ\cdots\circ\bari(\leng_{r_{1}}(\di \ri S)),
        \end{split}
    \end{align}
    where $\ri S$ denotes the $\gari$-inverse $\invgari(S)$ of $S$ and the $s=0$ terms in the above sums are understood as the identity.
\end{theorem}

\begin{lemma}\label{lem:der_ari}
    The linear operator $\der$ gives a derivation with respect to the binary operator $\preari$. Consequently it also gives a derivation on the Lie algebra $\ARI$.
\end{lemma}
\begin{proof}
    Let $A$ and $B$ be two elements of $\BIMU$. Since $\der$ gives a derivation of the associative algebra $\BIMU$ as
    \begin{align}
        &(\mmu(\der(A),B)+\mmu(A,\der(B)))(w_{1}\cdots w_{r})\\
        &=\sum_{i=0}^{r}\der(A)(w_{1}\cdots w_{i})B(w_{i+1}\cdots w_{r})+\sum_{i=0}^{r}A(w_{1}\cdots w_{i})\der(B)(w_{i+1}\cdots w_{r})\\
        &=\sum_{i=0}^{r}iA(w_{1}\cdots w_{i})B(w_{i+1}\cdots w_{r})+\sum_{i=0}^{r}(r-i)A(w_{1}\cdots w_{i})B(w_{i+1}\cdots w_{r})\\
        &=r\sum_{i=0}^{r}A(w_{1}\cdots w_{i})B(w_{i+1}\cdots w_{r})\\
        &=\der(\mmu(A,B))(w_{1}\cdots w_{r}),
    \end{align}
    it suffices to show the identity
    \[\der(\arit(B)(A))=\arit(\der(B))(A)+\arit(B)(\der(A))\]
    in the case where $B\in\LU$. Indeed we can show the equality
    \[\der(\amit(B)(A))=\amit(\der(B))(A)+\amit(B)(\der(A))\]
    by a simple computation as
    \begin{align}
        &(\amit(\der(B))(A)+\amit(B)(\der(A)))(\bw)\\
        &=\sum_{\substack{\bw=\ba\bb\bc\\ \bc\neq\emp}}\left(A(\ba\ful{\bb}\bc)\der(B)(\bb\flr{\bc})+\der(A)(\ba\ful{\bb}\bc)B(\bb\flr{\bc})\right)\\
        &=\sum_{\substack{\bw=\ba\bb\bc\\ \bc\neq\emp}}\left(\ell(\bb)\cdot A(\ba\ful{\bb}\bc)B(\bb\flr{\bc})+(\ell(\ba)+\ell(\bc))A(\ba\ful{\bb}\bc)B(\bb\flr{\bc})\right)\\
        &=\ell(\bw)\sum_{\substack{\bw=\ba\bb\bc\\ \bc\neq\emp}}A(\ba\ful{\bb}\bc)B(\bb\flr{\bc})\\
        &=\der(\amit(B)(A))(\bw),
    \end{align}
    as well as the identity
    \[\der(\anit(B)(A))=\anit(\der(B))(A)+\anit(B)(\der(A))\]    
    by a quite similar method.
\end{proof}

In the following, we denote by $L(S)$ and $R(S)$ the right-hand sides of \eqref{eq:main_theorem}, respectively.

\begin{lemma}\label{lem:key}
    For any $S\in\MU$, we have identitites
    \[[\der,L(S)]=L(S)\circ\bari(\di S)\]
    and
    \[[\der,R(S)]=-\bari(\di S)\circ R(S)\]
    of linear operators on $\LU$.
\end{lemma}
\begin{proof}
    %Let $L_{+}(S)$ denotes the operator $L(S)-\id$.
    We shall compute the bracket $[\der,L(S)]$: since Lemma \ref{lem:der_ari} holds, we have
    \begin{align}
        [\der,L(S)]      
        &=\begin{multlined}[t]\sum_{s=1}^{\infty}\sum_{r_{1},\ldots,r_{s}\ge 1}\frac{1}{r_{1}\cdots (r_{1}+\cdots+r_{s})}\\
        \cdot [\der,\bari(\leng_{r_{1}}(\di S))\circ\cdots\circ\bari(\leng_{r_{s}}(\di S))]\end{multlined}\\
        &=\begin{multlined}[t]\sum_{s=1}^{\infty}\sum_{r_{1},\ldots,r_{s}\ge 1}\frac{1}{r_{1}\cdots (r_{1}+\cdots+r_{s})}\\
        \cdot\sum_{i=1}^{s}\bari(\leng_{r_{1}}(\di S))\circ\cdots\circ\bari(\leng_{r_{i-1}}(\di S))\\
        \circ\bari((\der\circ\leng_{r_{i}})(\di S))\circ\bari(\leng_{r_{i+1}}(\di S))\circ\cdots\circ\bari(\leng_{r_{s}}(\di S))\end{multlined}\\
        &=\sum_{s=1}^{\infty}\sum_{r_{1},\ldots,r_{s}\ge 1}\sum_{i=1}^{s}\frac{r_{i}}{r_{1}\cdots (r_{1}+\cdots+r_{s})}\bari(\leng_{r_{1}}(\di S))\circ\cdots\circ\bari(\leng_{r_{s}}(\di S))\\
        &=\begin{multlined}[t]\sum_{s=1}^{\infty}\sum_{r_{1},\ldots,r_{s-1}\ge 1}\frac{r_{i}}{r_{1}\cdots (r_{1}+\cdots+r_{s-1})}\bari(\leng_{r_{1}}(\di S))\circ\cdots\circ\bari(\leng_{r_{s-1}}(\di S))\\\circ\sum_{r_{s}\ge 1}\bari(\leng_{r_{s}}(\di S))\end{multlined}\\
        &=L(S)\circ\bari(\di S).
    \end{align}
    The latter equality is shown in a similar manner.
\end{proof}

\begin{lemma}\label{lem:key2}
    It holds that
    \[[\der,\adari(S)]=\adari(S)\circ\bari(\di S)=-\bari(\di\ri S)\circ\adari(S)\]
    for any $S\in\MU$.
\end{lemma}
\begin{proof}
    Since $\adari$ is the adjoint action of $\GARI$ on $\ARI$, we have
    \begin{align}
        \adari(1+\ep\di S)
        &=\adari(\gari(\ri S,\exp(\ep\der)(S)))\\
        &=\adari(\ri S)\circ\adari(\exp(\ep\der)(S))\\
        &=\adari(S)^{-1}\circ\adari(\exp(\ep\der)(S))
    \end{align}
    by \eqref{eq:gari_dilator}.
    On the other hand, $\adari(1+\ep\di S)=\id+\ep\bari(\di S)$ holds by the general theory of adjoint actions. Thus we obtain
    \begin{align}
        \adari(S)+\ep\adari(S)\circ\bari(\di S)
        &=\adari(S)\circ(\id+\ep\bari(\di S))\\
        &=\adari(S)\circ\adari(1+\ep\di S)\\
        &=\adari(\exp(\ep\der)(S)).
    \end{align}
    Since Lemma \ref{lem:der_ari} holds, we see that $\exp(\ep\der)$ behaves distributively for $\adari$ as
    \begin{align}
        \exp(\ep\der)\circ\adari(\expari(L))
        &=\sum_{n=0}^{\infty}\frac{1}{n!}\exp(\ep\der)\bari(L)^{\circ n}\\
        &=\sum_{n=0}^{\infty}\frac{1}{n!}\bari(\exp(\ep\der)(L))^{\circ n}\circ\exp(\ep\der)\\
        &=\adari((\expari\circ\exp(\ep\der))(L))\circ\exp(\ep\der)\\
        &=\adari((\exp(\ep\der)\circ\expari)(L))\circ\exp(\ep\der)
    \end{align}
    for any $L\in\ARI$. Note that the final equality follows from Lemma \ref{lem:der_ari} and the expression
    \[\expari(L)=\sum_{n=0}^{\infty}\frac{1}{n!}\preari(\underbrace{L,\cdots\preari(L,L}_{n})\cdots).\]
    Hence we obtain
    \begin{align}
        \adari(S)+\ep\adari(S)\circ\bari(\di S)
        &=\adari(\exp(\ep\der)(S))\\
        &=\exp(\ep\der)\circ\adari(S)\circ\exp(-\ep\der)\\
        &=(\id+\ep\der)\circ\adari(S)\circ(\id-\ep\der)\\
        &=\adari(S)+\ep[\der,\adari(S)].
    \end{align}
    Comparing the coefficient of $\ep$, we obtain the first desired equality.
    For the latter, in a similar way we can show
    \begin{align}
    \adari(S)-\ep\bari(\di\ri S)\adari(S)
    &=\adari(1-\ep\di\ri S)\circ\adari(S)\\
    &=\adari(S)\circ\adari(\exp(-\ep\der)(\ri S))\circ\adari(S)\\
    &=\adari(S)\circ\exp(-\ep\der)\circ\adari(\ri S)\circ\exp(\ep\der)\circ\adari(S)\\
    &=\adari(S)\circ\exp(-\ep\der)\circ\adari(\ri S)\circ\adari(\exp(\ep\der)(S))\circ\exp(\ep\der)\\
    &=\adari(S)\circ\exp(-\ep\der)\circ\adari(\gari(\ri S,\exp(\ep\der)(S)))\circ\exp(\ep\der)\\
    &=\adari(S)\circ\exp(-\ep\der)\circ\adari(1+\ep\di S)\circ\exp(\ep\der)\\
    &=\adari(S)\circ(\id-\ep\der)\circ(\id+\ep\bari(\di S))\circ(\id+\ep\der)\\
    &=\adari(S)\circ(\id+\ep\bari(\di S)).
    \end{align}
\end{proof}

\begin{proof}[Proof of Theorem \ref{thm:main2}]
    We show that the equality $[\der,f]=f\circ\bari(\di S)$ and the initial condition
    \[f|_{\BIMU_{r}}\equiv \id~\mod~\BIMU_{\ge r+1}\quad (r\ge 1)\]
    determines the linear operator $f$ on $\LU$. Due to Lemmas \ref{lem:key} and \ref{lem:key2}, this assertion deduces Theorem \ref{thm:main2}. For integers $r$ and $s$, we put
    \[f_{r;s}\coloneqq\leng_{r+s}\circ f|_{\BIMU_{r}}\colon\BIMU_{r}\to\BIMU_{r+s}.\]
    We shall check that $f$ is inductively determined, that is, $f_{r;s}$ is expressed in terms of $f_{r';j}$'s for some $r'$ and $j\in\{0,1,\ldots,s-1\}$. Let $r$ be a positive integer and $A\in\BIMU_{r}$. Then we have
    \begin{align}
        [\der,f](A)
        &\equiv \der(f_{r;0}(A)+\cdots+f_{r;s}(A))-(f_{r;0}+\cdots+f_{r;s})(\der(A))\\
        &=r f_{r;0}(A)+(r+1)f_{r;1}(A)+\cdots +(r+s)f_{r;s}(A)-r(f_{r;0}+\cdots+f_{r;s})(A)\\
        &=(f_{r;1}+2f_{r;2}+\cdots+sf_{r;s})(A),
    \end{align}
    modulo $\BIMU_{\ge r+s+1}$, while it holds that
    \begin{align}
        (f\circ\bari(\di S))(A)
        &=\sum_{k\ge 1}f(\ari(\leng_{k}(\di S),A))\\
        &=\sum_{k\ge 1}\sum_{t\ge 1}f_{k+r,t}(\ari(\leng_{k}(\di S),A))\\
        &\equiv \sum_{\substack{k,t\ge 1\\ k+t\ge s}}f_{k+r,t}(\ari(\leng_{k}(\di S),A))\\
    \end{align}
    modulo the same space. Hence, by the equality $[\der,f]=f\circ\bari(\di S)$, the data $f_{i;j}$ ($i\ge 1$ and $0\le j<s$) determine $f_{r;s}$ for all $r\ge 1$. A similar method shows that $[\der,g]=-\bari(\di\ri S)\circ g$ also determines $g$.
\end{proof}

\section{Brown's lifting}
\subsection{The ingredient $\psi_{0}$ as a dilator}
\begin{definition}
    We define $\psi_{0}\in\lcL$ by letting its length $d$ component be 
    \begin{multline}
    \psi_{0}(x_{1},\ldots,x_{d})\\
    \coloneqq\binom{d+1}{2}^{-1}\left(\frac{d}{x_{1}\cdots x_{k}}+\sum_{k=1}^{d-1}\frac{d-k}{(-x_{k})(x_{1}-x_{k})\cdots (x_{k-1}-x_{k})(x_{k+1}-x_{k})\cdots (x_{d}-x_{k})}\right).
    \end{multline}
\end{definition}
We recall the definition of the mutually-conjugate pair of polar flexion units
\[\Pa\binom{u_{1}}{v_{1}}\coloneqq\frac{1}{u_{1}},\quad\Pi\binom{u_{1}}{v_{1}}\coloneqq\frac{1}{v_{1}}.\]
Denote by $\mpar$ the secondary bimould $\fess$ for $\fE=\Pa$.
\begin{lemma}\label{lem:darapir}
    Under the inclusion $\lcL\subseteq\LU^{\ucst}$, we have
    \begin{equation}\label{eq:darapir}
    \frac{1}{2}\anti(\psi_{0})=\dara\pir\coloneqq\swap(\diri\mpar).
    \end{equation}
\end{lemma}
\begin{proof}
    First we recall that, for each power series $f$ and a mutually-conjugate pair $(\fE,\fO)$ of flexion units, the $\gari$-dilator of the associated bimould $\fSe(f)$ is given by 
    \[\di\fSe(f)=\fTe(f)=\sum_{r=1}^{\infty}\gamma_{r}^{f}\fre_{r}\]
    by virtue of \cite[Lemma 6.16]{kawamura25}. Combining it and the explicit formula \cite[Lemma 6.8]{kawamura25} for $\dro_{r}=\swap(\fre_{r})$, we obtain
    \begin{align}
        &\dara\pir(w_{1}\cdots w_{r})\\
        &=\swap(\diri\mpar)(w_{1}\cdots w_{r})\\
        &=\swap(\di\fSe(\re^{-1})|_{\fE=\Pa})(w_{1}\cdots w_{r})\\
        &=\swap(\fTe(\re^{-1})|_{\fE=\Pa})(w_{1}\cdots w_{r})\\
        &=\gamma_{r}^{\re^{-1}}\dro_{r}|_{\fO=\Pi}(w_{1}\cdots w_{r})\\
        &=\frac{1}{r(r+1)}\sum_{i=1}^{r}(r+1-i)\pic((w_{1}\cdots w_{i-1})\flr{w_{i}})\Pi(\ful{w_{1}\cdots w_{i-1}}w_{i}\fur{w_{i+1}\cdots w_{r}})\pic(\fll{w_{i}}(w_{i+1}\cdots w_{r})),
    \end{align}
    where $\pic$ denotes the primary bimould $\foz$ for $\fO=\Pi$, that is, we put
    \[\pic(w_{1}\cdots w_{r})\coloneqq\invmu(1-\Pi)(w_{1}\cdots w_{r})=\Pi(w_{1})\cdots \Pi(w_{r}).\]
    On the other hand, by definition, the bimould $\psi_{0}$ is expressed as
    \begin{align}
        &\frac{1}{2}\anti(\psi_{0})(w_{1}\cdots w_{r})\\
        &=\frac{1}{2}\psi_{0}(w_{r},\ldots,w_{1})\\
        &=\begin{multlined}[t]\frac{1}{r(r+1)}\left(\frac{r}{v_{r}\cdots v_{1}}\phantom{\sum_{k=2}^{r}}\right.\\\left.+\sum_{k=2}^{r}\frac{k-1}{(-v_{k})(v_{r}-v_{k})\cdots (v_{k+1}-v_{k})(v_{k-1}-v_{k})\cdots (v_{1}-v_{k})}\right)\end{multlined}\\
        &=\begin{multlined}[t]\frac{1}{r(r+1)}\left(r\pic(w_{1}\cdots w_{r})\phantom{\sum_{k=2}^{r}}\right.\\\left.-\sum_{k=2}^{r}(k-1)\pic((w_{1}\cdots w_{k-1})\flr{w_{k}})\Pi(\ful{w_{1}\cdots w_{k-1}}w_{k}\fur{w_{k+1}\cdots w_{r}})\pic(\fll{w_{k}}(w_{k+1}\cdots w_{r}))\right).\end{multlined}
    \end{align}
    under the abbreviation $w_{i}=\binom{u_{i}}{v_{i}}$. Hence we obtain
    \begin{align}
        &(\dara\pir-\frac{1}{2}\anti(\psi_{0}))(w_{1}\cdots w_{r})\\
        &=\begin{multlined}[t]\frac{1}{r(r+1)}\left(r\Pi(w_{1}\fur{w_{2}\cdots w_{r}})\pic(\fll{w_{1}}(w_{2}\cdots w_{r}))-r\pic(w_{1}\cdots w_{r})\phantom{\sum_{k=2}^{r}}\right.\\\left.+r\sum_{i=2}^{r}\pic((w_{1}\cdots w_{i-1})\flr{w_{i}})\Pi(\ful{w_{1}\cdots w_{i-1}}w_{i}\fur{w_{i+1}\cdots w_{r}})\pic(\fll{w_{i}}(w_{i+1}\cdots w_{r}))\right)\end{multlined}\\
        &=\begin{multlined}[t]\frac{1}{r(r+1)}\left(-r\pic\binom{u_{2},\ldots,u_{r}}{v_{2}-v_{1},\ldots,v_{r}-v_{1}}\Pi\binom{-u_{1}-\cdots-u_{r}}{-v_{1}}-r\pic(w_{1}\cdots w_{r})\phantom{\sum_{k=2}^{r}}\right.\\\left.-r\sum_{i=2}^{r}\pic\binom{u_{1},\ldots,u_{i-1}}{v_{1}-v_{i},\ldots,v_{i-1}-v_{i}}\Pi\binom{-u_{1}-\cdots-u_{r}}{-v_{i}}\pic\binom{u_{i+1},\ldots,u_{r}}{v_{i+1}-v_{i},\ldots,v_{r}-v_{i}}\right)\end{multlined}\\
        &=\frac{1}{r(r+1)}\left(-r\push^{\circ r}(\pic)(w_{1}\cdots w_{r})-r\pic(w_{1}\cdots w_{r})-r\sum_{i=2}^{r}\push^{\circ (r+1-i)}(\pic)(w_{1}\cdots w_{r})\right)\\
        &=\frac{1}{r(r+1)}\left(-r\push^{\circ r}(\pic)(w_{1}\cdots w_{r})-r\sum_{i=2}^{r+1}\push^{\circ (r+1-i)}(\pic)(w_{1}\cdots w_{r})\right).
    \end{align}
    Since $\Pi$ is a flexion unit, the bimould $\pic$ is $\push$-neutral due to \cite[Proposition 4.2 (3)]{kawamura25}. It yields that the above expression is equal to $0$.
\end{proof}

\begin{remark}
    In \cite{ecalle15}, Ecalle used the notation $\mathsf{da}S$ (resp.~$\mathsf{ra}S$) to indicate the $\gira$-dilator (resp.~$\gira$-inverse) of each bimould $S\in\MU$. Though we defined $\dara\pir$ as the swappee of $\diri\mpar$, but one can show that the general equality $\swap(\dara S)=\diri(\swap(S))$ for any $S\in\MU$.
\end{remark}

From the coincidence \eqref{eq:darapir}, we obtain the following result, whose proof is omitted in the original paper \cite{brown17}.

\begin{corollary}[{\cite[Theorem 14.2]{brown17}}]
    We have $\psi_{0}\in\ds_{\cQ}$.
\end{corollary}
\begin{proof}
    It suffices to show $\diri\mpar=(1/2)(\swap\circ\anti)(\psi_{0})\in\ARI_{\baral/\baril}$. The parity condition is obvious by definition. The construction of $\mpar$ and \cite[Proposition 6.13]{kawamura25} show that $\mpar$ is symmetral and consequently $\ri\mpar\in\GARI_{\as}$. Hence we have $\diri\mpar\in\ARI_{\al}$ by \cite[Proposition A.3]{kawamura252}. The alternility of the swappee is a special case of \cite[Theorem A.6]{kawamura252}.
\end{proof}

\begin{remark}
    Brown \cite{brown17} denoted by $s_{d}$ the rational function obtained by removing the binomial cofficient factor from the definition of $\psi_{0}^{(d)}$. Then we see that, the Witt-type identity $\{s_{m},s_{n}\}=(m-n)s_{m+n}$ mentioned in \cite[\S 5.2]{mt19}, is a special case of \cite[Proposition 6.9]{kawamura25}, because each $s_{d}$ is exactly $\anti(\dro_{d})|_{\fO=\Pi}$ as shown in the proof of Lemma \ref{lem:darapir}.
\end{remark}

\subsection{Brown's lifting as a dimorphic transportation}
\begin{definition}
    Let $f=f^{(d)}$ be in $\lcL$.
    Then we set
    \[\chi_{B}(f)^{(i)}\coloneqq\begin{cases}
        f & \text{if }i=d,\\
        0 & \text{if }0\le i<d,
    \end{cases}\]
    and define its higher length parts by the recursive equation
    \begin{equation}\label{eq:brown_lift}
    \chi_{B}(f)^{(d+r)}=\frac{1}{2r}\sum_{i=1}^{r}\{\psi_{0}^{(i)},\chi_{B}(f)^{(d+r-i)}\}.
    \end{equation}
\end{definition}

Here we present a proof of Theorem \ref{thm:main}.
Repeatedly applying \eqref{eq:brown_lift}, we see that
\[\chi_{B}(f)^{(d+r)}=\sum_{\substack{r_{1},\ldots,r_{s}\ge 1\\ r_{1}+\cdots+r_{s}=r \\ s\ge 1}}\frac{1}{(r_{1}+\cdots+r_{s})\cdots (r_{s-1}+r_{s})r_{s}}\left\{\frac{\psi_{0}^{(r_{1})}}{2},\left\{\cdots\left\{\frac{\psi_{0}^{(r_{s})}}{2},f\right\}\cdots\right\}\right\}\]
holds for every $r\ge 1$.
We know that $f^{\sharp}\in\cV$ as a consequence of the assumption $f\in\ls_{\cQ}$, and also that $\psi_{0}$ belongs to $\cV$ since so is $\dara\pir$ due to \cite[Lemma 7.8]{kawamura25}.
Then Proposition \ref{cor:mt45} deduces
\begin{align}
    (\chi_{B}(f)^{(d+r)})^{\sharp}
%    &=(\swap\circ\anti)(\chi_{B}(f)^{(d+r)})\\
    &=\sum_{\substack{r_{1},\ldots,r_{s}\ge 1\\ r_{1}+\cdots+r_{s}=r \\ s\ge 0}}\frac{1}{(r_{1}+\cdots+r_{s})\cdots (r_{s-1}+r_{s})r_{s}}\left\{\frac{\psi_{0}^{(r_{1})}}{2},\cdots\left\{\frac{\psi_{0}^{(r_{s})}}{2},f\right\}\cdots\right\}^{\sharp}\\
    &=\sum_{\substack{r_{1},\ldots,r_{s}\ge 1\\ r_{1}+\cdots+r_{s}=r \\ s\ge 0}}\frac{1}{(r_{1}+\cdots+r_{s})\cdots (r_{s-1}+r_{s})r_{s}}\left\{\frac{(\psi_{0}^{(r_{1})})^{\sharp}}{2},\cdots\left\{\frac{(\psi_{0}^{(r_{s})})^{\sharp}}{2},f^{\sharp}\right\}_{\ari}\cdots\right\}_{\ari}\\
    &=\sum_{\substack{r_{1},\ldots,r_{s}\ge 1\\ r_{1}+\cdots+r_{s}=r \\ s\ge 0}}\frac{1}{(r_{1}+\cdots+r_{s})\cdots (r_{s-1}+r_{s})r_{s}}\ari\left(\cdots\ari\left(f^{\sharp},\frac{(\psi_{0}^{(r_{s})})^{\sharp}}{2}\right),\cdots,\frac{(\psi_{0}^{(r_{1})})^{\sharp}}{2}\right)\\
    &=\sum_{\substack{r_{1},\ldots,r_{s}\ge 1\\ r_{1}+\cdots+r_{s}=r \\ s\ge 0}}\frac{(-1)^{s}}{(r_{1}+\cdots+r_{s})\cdots (r_{s-1}+r_{s})r_{s}}\ari\left(\frac{(\psi_{0}^{(r_{1})})^{\sharp}}{2},\cdots\ari\left(\frac{(\psi_{0}^{(r_{s})})^{\sharp}}{2},f^{\sharp}\right)\cdots\right).
\end{align}
Moreover, by Lemma \ref{lem:darapir}, we know that
\[\frac{1}{2}(\psi_{0}^{(r_{j})})^{\sharp}=\frac{1}{2}(\leng_{r_{j}}\circ\swap\circ\anti\circ\anti)(\dara\pir)=\leng_{r_{j}}(\diri\mpar).\]
Therefore we obtain
\begin{align}
    \chi_{B}(f)^{\sharp}
    &=\sum_{r\ge 0}\chi_{B}(f)^{(d+r)}\\
    &=\sum_{\substack{r_{1},\ldots,r_{s}\ge 1\\ s\ge 0}}\frac{(-1)^{s}}{(r_{1}+\cdots+r_{s})\cdots (r_{s-1}+r_{s})r_{s}}\ari(\leng_{r_{1}}(\diri\mpar),\cdots\ari(\leng_{r_{s}}(\diri\mpar),f^{\sharp})\cdots)\\
    &=\adari(\mpar)(f^{\sharp}),
\end{align}
from the second equality of Theorem \ref{thm:main2}.

\begin{remark}
    Kimura--Tasaka's conjecture \cite[予想 3.5]{tasaka} considers the conditions under which a lifting $\chi_{\psi}$ constructed by replacing $\psi_{0}$ with another function $\psi$ in Brown's lifting becomes an isomorphism such as $\chi_{B}$. Although we do not know what conditions are appropriate, it can be shown, by following the exact same method as the proof of Theorem \ref{thm:main}, that the resulting lifting $\chi_{\psi}$ is the $\anti\circ\swap$-conjugation of $\adari(X)$ for the bimould $X$ satisfying $\diri X=\psi$.
\end{remark}

\end{document}